\documentclass[11pt,leqno,twoside]{amsart}
\usepackage{amssymb,amsmath,amsthm,soul,color}
\usepackage{t1enc}
\usepackage[cp1250]{inputenc}
\usepackage{a4,indentfirst,latexsym,}
\usepackage{graphics}
\usepackage{mathrsfs}
\usepackage{cite,enumitem,graphicx}
\usepackage[colorlinks=true,urlcolor=blue,
citecolor=red,linkcolor=blue,linktocpage,pdfpagelabels,
bookmarksnumbered,bookmarksopen]{hyperref}
\usepackage[english]{babel}
\usepackage[left=2.61cm,right=2.61cm,top=2.72cm,bottom=2.72cm]{geometry}
\usepackage[metapost]{mfpic}
\usepackage[hyperpageref]{backref}
\usepackage[colorinlistoftodos]{todonotes}
\usepackage[normalem]{ulem}

\makeatletter
\providecommand\@dotsep{5}
\def\listtodoname{List of Todos}
\def\listoftodos{\@starttoc{tdo}\listtodoname}
\makeatother

\numberwithin{equation}{section}
\newtheorem{theorem}{Theorem}[section]

\newtheorem{remark}{Remark}

\def\R{\mathbb{R}}

\title[Multiple solutions for some  strongly degenerate elliptic equations]
{Multiple solutions for some  \\ strongly degenerate 
second order elliptic equations}

\author[J. R. Santos Jr.]
{Jo\~ao R. Santos Junior
}
\author[G. Siciliano]{Gaetano Siciliano$^{*}$ 
 }

\address[J. R. Santos Jr.]
{\newline\indent Faculdade de Matem\'atica
\newline\indent 
Instituto de Ci\^{e}ncias Exatas e Naturais
\newline\indent 
Universidade Federal do Par\'a
\newline\indent
Avenida Augusto corr\^{e}a 01, 66075-110, Bel\'em, PA, Brazil}
\email{\href{mailto: joaojunior@ufpa.br }{joaojunior@ufpa.br}
}
\address[G. Siciliano]
{\newline\indent Departamento de Matem\'atica
\newline\indent 
Instituto de Matem\'atica e Estat\'istica
\newline\indent 
 Universidade de S\~ao Paulo 
\newline\indent 
Rua do Mat\~ao 1010,  05508-090, S\~ao Paulo, SP, Brazil }
\email{\href{mailto:sicilian@ime.usp.br}{sicilian@ime.usp.br}
}

\thanks{Jo\~ao R. Santos was partially
supported by CNPq 306503/2018-7, Brazil.
Gaetano Siciliano  was partially supported by
Fapesp 2018/17264-4, Capes and CNPq 304660/2018-3, Brazil. \\
($*$) Corresponding author}
\subjclass[2010]{ 
35J50,  	
35J57, 
	35J70. 
	}
\keywords{ Elliptic equations, degenerate operators, vanishing solutions}

\begin{document}

\maketitle
\begin{abstract}
We consider a boundary value problem in a bounded domain involving 
a degenerate operator of the form
$$L(u)=-\textrm{div} (a(x)\nabla u)$$
and a suitable nonlinearity $f$.
The function $a$  vanishes on smooth 1-codimensional submanifolds of $\Omega$
where it is not allowed to be $C^{2}$.
By using  weighted Sobolev spaces we are still able to find existence of solutions
which vanish, in the trace sense, on the set where $a$ vanishes.
\end{abstract}

\bigskip

\maketitle
\begin{center}
\begin{minipage}{12cm}
\tableofcontents
\end{minipage}
\end{center}

\bigskip

\section{Introduction}


In this paper we are interested in the existence of ``suitable'' solutions
for a  degenerate nonlinear elliptic equation of second order in a bounded
and smooth domain in $\R^{N}$ with homogeneous Dirichlet boundary condition.
More specifically the equation under study  is driven by the operator 
$$L(u)=-\textrm{div} (a(x)\nabla u)$$
where, $a:\overline\Omega \to [0,+\infty)$, among other assumptions, is a continuous function such that $a(x)>0$
in the whole $\Omega$ except for suitable $1-$codimensional submanifolds
contained in $\Omega$ where it vanishes.
Hence  the ellipticity of $L$ is broken somewhere in $\overline\Omega$.
This kind of operator is also called   {\sl degenerate} due to the fact that 
$a^{-1}$ is unbounded.

Degenerate operators appear in many situations. Indeed 
it is known that many physical phenomena are described by degenerate evolution equations,
where the degeneracy can be due to the 
vanishing of the time derivative coefficient or to the vanishing of the diffusion coefficient.
In this context there is a strong connexion between degenerate 2nd order differential operators and
Markov processes: roughly speaking these operators describe a diffusion phenomena
of Markovian particle which moves until it reaches the set where the absorption takes place
and here the particle ``dies''. Because of this fact, degenerate equations are 
appropriate to describe
 fluid diffusion in nonhomogeneous  porous media
 taking into account saturation and porosity of the medium.
For more  applications and problems involving degenerate operators
one can see e.g \cite{C,CR,EM, EM2, MS} and the references therein.

Mathematically speaking, for degenerate partial differential equations, i.e., equations with various types of singularities in the coefficients, it is natural to look for solutions in weighted Sobolev spaces. A class of weights, which is particularly well understood, is the class of $A_{p}-$weights (or Muckenhoupt class) that was introduced by B. Muckenhoupt.
 The importance of this class is that 
powers of distance to submanifolds of  $\mathbb R^{N}$ often belong to $A_{p}$ (see \cite{KJF})
and these weight have found many useful applications also in harmonic analysis (see \cite{T}).
However there are also many other interesting examples of weights (see \cite{HKM} for $p$-admissible weights).
For some references on this subject see also  \cite{FJK,FS,FKS,K},
and for other applications of weighted Sobolev spaces see also \cite{SK}.

\medskip


To motivate the choice of the problem under study let us see the following example.
 Suppose 
additionally
that $f:\mathbb R\to \mathbb R$ is a continuous function such that $f(s)=0$
if and only if $s=0$. 
Let $\Omega\subset \mathbb R^{N}, N\geq2$
be a smooth and bounded domain and assume that $a\in C^{1}(\overline\Omega)$ 
is a positive function with  $a^{-1}(0)$ which is a regular connected submanifold
compactly contained in $\Omega$ 
and such that $\nabla a(x)=0$ for any $x\in a^{-1}(0)$. 
%
Consider the problem 
\begin{equation}\label{eq:distrib}
- \textrm{div}(a(x)\nabla u)= f(u)  \ \ \ \textrm{in} \ \  \mathcal D'(\Omega).
\end{equation}

Following \cite{PS} we say that $u_{*}\in \mathcal D'(\Omega)$
is a solution if $u_{*}\in  C^{1}(\Omega)$ and
the equation is 
satisfied in the sense of distribution, i.e. 
$$\int_{\Omega}  a(x) \nabla u_{*}\nabla \varphi = \int_{\Omega} f(u_{*})\varphi \quad \forall \varphi\in C^{\infty}_{c}(\Omega).$$
But then from \eqref{eq:distrib} it follows,  that 
$$-\nabla a(x)\nabla u_{*}-a(x)\Delta u_{*} = f(u_{*}) \quad \text{in } \mathcal D'(\Omega)$$
and since $f(u_{*})$ and $\nabla a(x) \nabla u_{*}$ are continuous functions, so is $a(x)\Delta u(x)$
(note that $a$ vanishes on a null set)
and we obtain
$$-\nabla a(x)\nabla u_{*}-a(x)\Delta u_{*} = f(u_{*}(x)) \quad \forall x\in \Omega.$$
From this identity 
we 
deduce
$$x\in a^{-1}(0) \Longrightarrow f(u_{*}(x))=0  \Longrightarrow 
u_{*}(x)=0.$$
 In other words, for such a problem, the solution is zero whenever $a$ is zero.

\medskip

Motivated by this fact we study in this paper the existence of weak solutions
for a degenerate elliptic operator in a bounded domain with homogeneous
Dirichlet boundary condition and with the additional condition that our solutions are zero
 (in the sense of trace) on the set where 
 $a$ vanishes.
More specifically the problem under study is the following.

Let $\Omega\subset\mathbb R^{N}, N\geq2$ be a smooth and bounded domain,
  $a\in C(\overline{\Omega})$, $a\geq 0$ and $f\in C(\R)$ are functions satisfying:
\begin{enumerate}[label=(a\arabic*),ref=a\arabic*]
\item\label{a1} $a^{-1}(0)=\cup_{l=1}^{k}\Gamma_l\subset\Omega$ is the disjoint union of a finite number $k$ of compact, connected, without boundary and $1$-codimensional smooth submanifolds $\Gamma_l$ of $\R^N$,
\item\label{a2} $a\in A_2$ (the standard Muckenhoupt class) and $1/a\in L^{t}(\Omega)$, for some $t>N/2$,
\end{enumerate}
and
\begin{enumerate}[label=(f\arabic*),ref=f\arabic*]
\item\label{f1} $f$ has a strict local minimum in $s=0$ with $f(0)=0$, and 
there exists $s_{\ast}>0$ such that  $f(s_{\ast})=0$ and $f>0$ in $(0,s_{\ast})$,
\item\label{f2} there exists $\gamma=\lim_{t\to 0^{+}} f(s)/t>0$ and $a_{M_j}:=\textrm{max}_{x\in\overline{D}_{j}}a(x)<\gamma/\lambda_1(D_j)$, where $\lambda_{1}(D_j)$ is the first eigenvalue of the Dirichlet Laplacian in $D_j$ and $D_j$ stands for any connected component of $\Omega\backslash a^{-1}\{0\}$.
\end{enumerate}
Consider the problem
 \begin{equation}\label{P}\tag{P}
\left \{ \begin{array}{ll}
- \textrm{div}(a(x)\nabla u)= f(u) & \mbox{in $\Omega$,}\medskip \\
u=0 & \mbox{on $\partial\Omega\cup a^{-1}(0)$}
\end{array}\right.
\end{equation}
The requirement that $u$ vanishes also on the set $a^{-1}(0)$ is motivated
 by the previous example.
 
 \medskip

A {\sl weak solution of \eqref{P}} is a function $u_{*}\in W^{1,1}_{0}(\Omega\setminus a^{-1}(0))\cap L^{\infty}(\Omega)$ such that
 $$\int_{\Omega}a(x)\nabla u_{*}\nabla \varphi =\int_{\Omega}f(u_{*})\varphi , \ \ \forall  \varphi\in C_{c}^{\infty}(\Omega\backslash a^{-1}(0)).
 $$
Note that, since $a\in C(\overline{\Omega})$ and $f\in C(\R)$, the 
above identity makes sense.
The choice of the space $W_{0}^{1,1}(\Omega\setminus a^{-1}(0))$ in place of the 
more common space $H^{1}_{0}(\Omega\setminus a^{-1}(0))$ is due to the fact that we do not
know if the gradient of the solution $u_{*}$ we find is in $L^{2}(\Omega \setminus a^{-1}(0))$.

\medskip

Before to continue, let us make few comments on the assumptions.
First of all, note that we are just assuming the continuity of $a$ and, 
in contrast to our motivating problem \eqref{eq:distrib}, the function $f$
is also allowed to vanish in many points (assumption \eqref{f1}); however
there is a relation of its first right derivative in zero with the function $a$
(assumption \eqref{f2}).

The class $A_{2}$ which appears in assumption \eqref{a2} is the Muckenhoupt class.
We prefer do not recall the right definition here
(see the next Section) but roughly speaking it gives a condition on the summability of
$a$ and $1/a$ and it seems the right class to work with and define reasonable
weighted Sobolev spaces for such a problem.
 
Finally it is worth to say that assumption \eqref{a1} appeared also in \cite{MPP}
where the authors study an operator
of type $\textrm{div}(A(x) \nabla u)$, for a suitable matrix $A$ which can vanish.
They are interested actually in establishing Poincar\'e type inequalities for such a degenerate operator.

\begin{remark}
It is easy to exhibits example of functions $a$ satisfying our assumptions.
Let $\Omega = B_{2}(0)$ be the  ball entered in $0$ in $\mathbb R^{N}, N\geq2$ of radius $2$.

Take a radial function whose profile in the radial variable has zeroes of order less then one, for example  
$$a(r) = 
\begin{cases}
\sqrt[3]{1-r^{2}} & \mbox { if } r\in[0,1],  \smallskip \\
\sqrt{(1-r)(r-2)} & \mbox { if } r\in(1,2]. \\
\end{cases}
$$
Then it is easy to check that $a\in A_{2}, 1/a\in L^{t}(\Omega)$
 for any $t\in[1,2N)$ and then \eqref{a2} holds. The function is of course not of class $C^{1}$
 where it vanishes.

\medskip
 
 Similarly, consider a  function  which is strictly positive in the center of the ball $\Omega$
 and whose radial profile  is $C^{1}$, with null derivative in the origin, and of type
 $$a(r) = 
\begin{cases}
{\rm smooth  \ and \ positive } & \mbox { if } r\in[0,1/5],  \smallskip \\
\displaystyle\frac{(r-1)^{2}}{\sqrt{|r-1|}}& \mbox { if } r\in(1/5,6/5],  \smallskip \\
{\rm smooth  \ and \ positive} & \mbox { if } r\in(6/5,11/5),  \smallskip \\
\displaystyle\frac{(r-2)^{2}}{\sqrt{|r-2|}}& \mbox { if } r\in[11/5,2].
\end{cases}
$$
It is easy to check that $a\in A_{2}, 1/a\in L^{t}(\Omega)$ for any $t\in [1,2N/3)$
and then  \eqref{a2} holds. 
Such a function is   $C^{1}$ in all $\Omega$, and  $C^{2}$ 
in $\Omega$ except where it vanishes.
 
 \medskip

Note however that functions that are $C^{2}$ where they vanish  are  not allowed by our hypothesis.
Indeed, if $a$ were  positive and of class $C^{2}$ in a neighbourhood of $x_{0}\in \Omega$
where $a(x_{0})=0$, then by the Taylor expansion,
 $$a(x) \leq C |x-x_{0}|^{2} \quad \text{in a neighbourhood $U_{x_{0}}$ of }  x_{0}.$$
 It follows 
 $$\frac{1}{a(x)^{N/2}} \geq \frac{C}{|x-x_{0}|^{N}} \quad \text{ in }U_{x_{0}}$$
 then $1/a \notin L^{N/2}(\Omega)$ and hence \eqref{a2} cannot be satisfied.
\end{remark}

\medskip

To state our main result let us fix some notations. 

%

Denote $\Gamma_{k+1}=\partial\Omega$.
Let $\pi_{0}(\Omega\setminus a^{-1}(0))$ be the usual
quotient space of  $\Omega\setminus a^{-1}(0)$
under the equivalence relation which identifies points
that can be joint with a continuous arch.
Then 
$\chi:=\textrm{card }\pi_{0}(\Omega\setminus a^{-1}(0))\geq1$
gives the number of  connected components of $\Omega\setminus a^{-1}(0)$.
Let us  write 
$$
\chi=
\sum_{i=1}^{m}j_i, \qquad  j_i\in \mathbb N, \ j_{1} \geq1,
$$
where $j_{i}$ stands for the number of subdomains  of $\Omega\backslash a^{-1}(0)$ whose boundary
is made  exactly by $i$ connected $1$-codimensional submanifolds of $\mathbb R^{N}$.
These domains are denoted with 
$\mathcal{A}^{(i)}_1, \mathcal{A}^{(i)}_2, \ldots, \mathcal{A}^{(i)}_{j_{i}}$. 
See the Figure \ref{fig1} and Figure \ref{fig2} for two examples in dimension two.

%
%
%

\begin{figure}[b]
\centering%
 \includegraphics[scale=0.5]{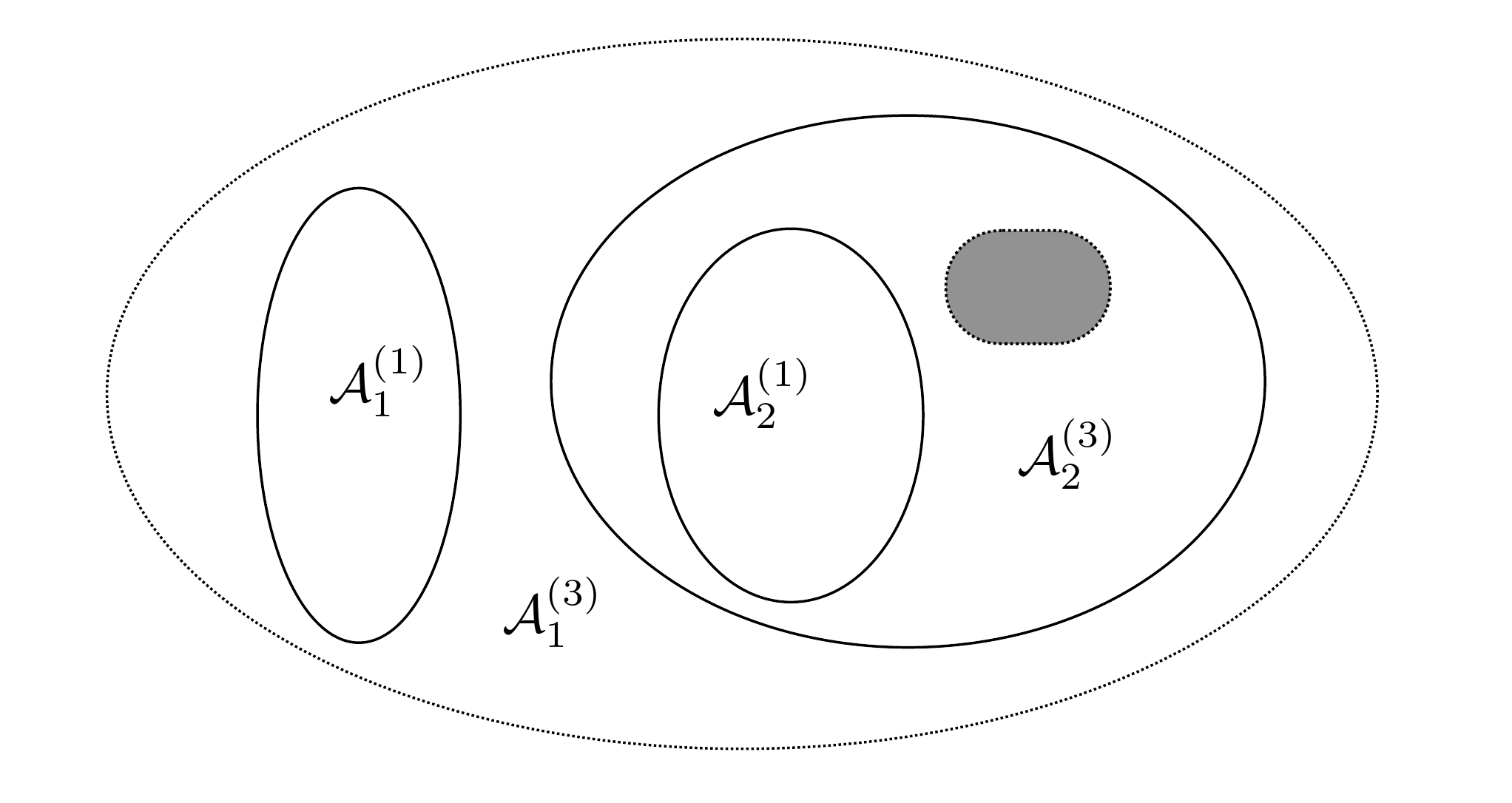}
\caption{\label{fig1}
Example  of a domain (with one grey hole) where $a^{-1}(0)=\sum_{i=1}^{3}\Gamma_{i}$.
In this case $\chi=4, j_{1}=2, j_{2}=0, j_{3}=2$.}
\end{figure}

\begin{figure}[h]
\centering%
\includegraphics[scale=0.5]{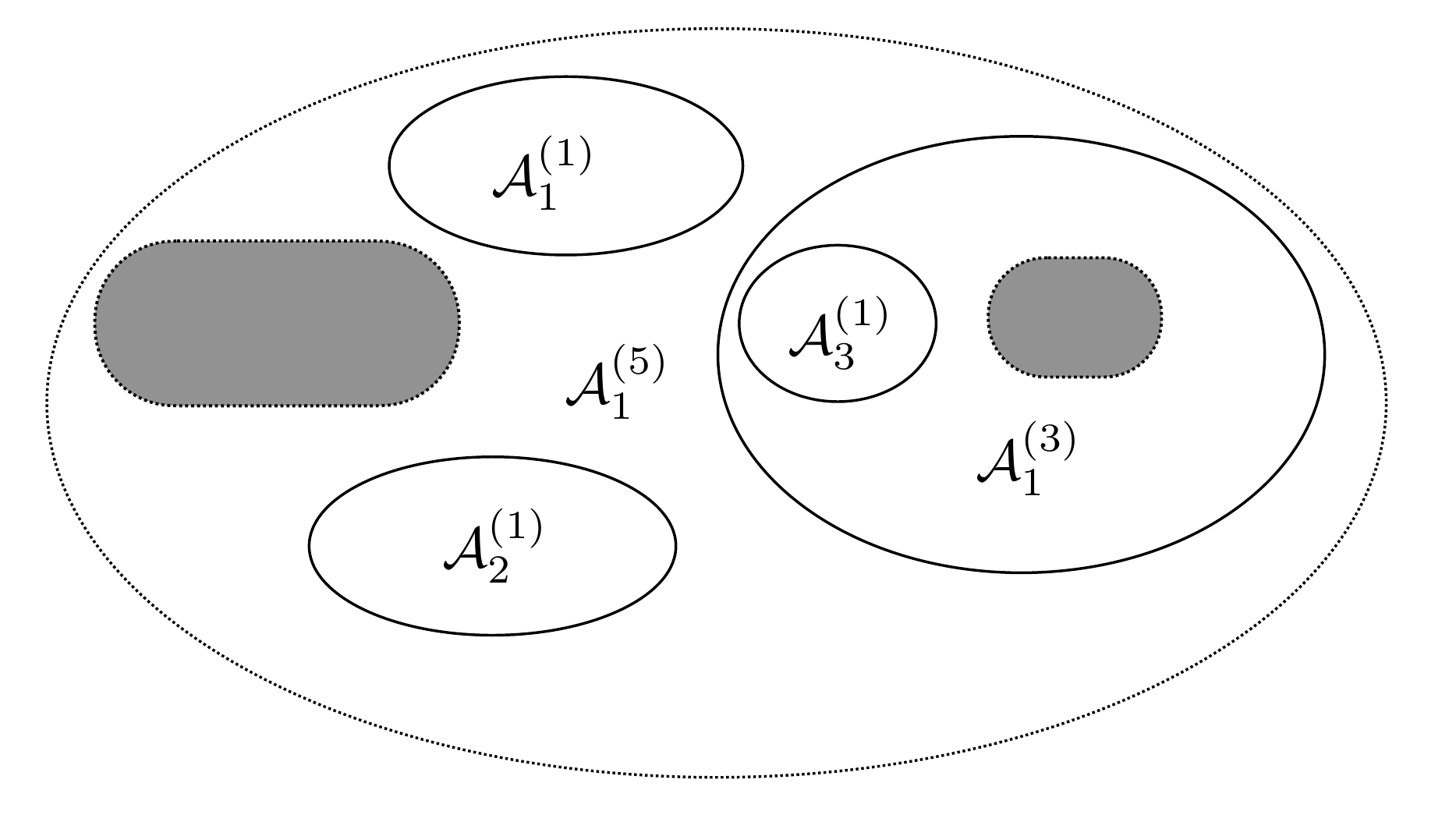}
\caption{\label{fig2}
Example  of a domain (with two grey holes) where $a^{-1}(0)=\sum_{i=1}^{4}\Gamma_{i}$.
In this case $\chi=5, j_{1}=3, j_{2}=0, j_{3}=1, j_{4}=0, j_{5}=1$.}
\end{figure}

\medskip

Our result states that the  number of solutions of \eqref{P} 
is related to $\chi.$

\begin{theorem}\label{th:main}
Suppose that \eqref{a1},\eqref{a2}, \eqref{f1}, \eqref{f2} hold. Then, problem \eqref{P} has at least $2^{\chi}-1$ nonnegative (and nontrivial) weak solutions. 
More specifically, the number of positive solutions with $n$ bumps, $n\in\{1,\ldots, \chi \}$, is given by 
the binomial coefficient $\frac{\chi!}{n!(\chi-n)!}.$
\end{theorem}



We point out that we will use variational methods to prove our result
and we will work in the weighted Sobolev space $H^{1}_{0}(\Omega,a)$;
so the solutions we find actually will belong to this space.


\medskip

The paper is organised as follows.  In the next Section \ref{sec:wellknown}
we recall some basic facts on weighted Sobolev spaces to establish the framework of our problem.
In Section \ref{sec:prelim} a suitable problem is solved which will be the main ingredient
to prove our main result in the last Section \ref{sec:main}.

%
\medskip\subsection*{Notations} As a matter of notations, 
in all the paper we denote with $W^{m,p}(\Omega)$ the usual Sobolev spaces. Whenever $p=2$
we use the notation $H^{m}(\Omega)$. Finally $H^{1}_{0}(\Omega)$ is the closure of the test functions
with respect to the norm in $H^{1}(\Omega)$.
Other notations will be introduced whenever we need.

\section{Some well known facts}\label{sec:wellknown}

In this section we will give some preliminary facts on suitable weighted Sobolev spaces we will use later.
For more details  and applications of weighted Sobolev spaces,
which is the right context to study degenerate elliptic operators,
we refer the reader to \cite{FKS,GU, KJF,K,M,KO}, for instance.

Along  this section

\begin{itemize}
\item[1.] $\Omega\subset\R^{N}$ is a smooth and bounded domain, and  \medskip
\item[2.]  $h:\Omega\to[0,+\infty)$ 
satisfies
$$
\sup\left(\frac{1}{|B|}\int_{B}h(x) \right)\left(\frac{1}{|B|}\int_{B}h(x)^{-\frac{1}{p-1}}\right)^{p-1}\leq C, \quad p>1,
$$
where the supremum is taken over all the balls $B\subset \Omega$.  In other words,
$h$ belongs to the so called Muckenhoupt class $A_p$ (see \cite{M}).
The right hand side of the  inequality above is known as the $A_p$-constant of $h$. 

\end{itemize}

For each $p\geq 1$,  $L^{p}(\Omega, h)$ is the Banach space of all measurable functions $u:\Omega\to\R$, for which
$$
|u|_{L^{p}(\Omega, h)}=\left(\int_{\Omega}h(x)|u|^{p}\right)^{1/p}<\infty.
$$
Whenever $h$ is in the $A_{p}$ class, $L^{p}(\Omega, h)\subset L^{1}_{loc}(\Omega)$ and then it makes sense to speak about weak derivatives and Sobolev spaces.
By definition, the weighted Sobolev space $H^{1}(\Omega, h)$ is the set of functions $u\in L^{2}(\Omega, h)$ such that the  (weak) derivatives of first order are all  in $L^{2}(\Omega, h)$. 
The (squared) norm in $H^{1}(\Omega, h)$ is
$$
\|u\|^{2}_{H^{1}(\Omega, h)}=\int_{\Omega}h(x) \left(|\nabla u|^{2}+|u|^{2}\right).
$$
It can be proved that $H^{1}(\Omega, h)$ is the closure if $C^{\infty}(\overline\Omega)$ with respect to the previous norm. As usual, $H^{1}_{0}(\Omega, h)$ is the closure of $C_{c}^{\infty}(\Omega)$ with respect to the 
 norm defined by
\begin{equation}\label{eq:norma0}
\|u\|^{2}_{H^{1}_{0}(\Omega, h)}=\int_{\Omega}h(x)|\nabla u|^{2}.
\end{equation}

Both $H^{1}(\Omega, h)$ and $H^{1}_{0}(\Omega, h)$ are Hilbert spaces
containing the positive and negative parts of each of their elements (see \cite[Corollary 2.1]{FKS}).
Since $h$ may vanish somewhere on $\overline \Omega$, the weighted Sobolev spaces
are not isomorphic the the ``usual'' ones.

\begin{theorem}(The weighted Sobolev inequality)
There exists positive constants $C_{\Omega}$ and $\delta$, such that for all $u\in C_{c}^{\infty}(\Omega)$ and $1\leq\theta\leq N/(N-1)+\delta$,
$$
|u|_{L^{2\theta}(\Omega, h)}\leq C_{\Omega}|\nabla u|_{L^{2}(\Omega, h)}.
$$
\end{theorem}
See \cite[Theorem 1.3]{FKS} for a proof. In particular from this results it hods that the quantity defined in \eqref{eq:norma0}  gives a norm on $H^{1}_{0}(\Omega, h)$
equivalent to $\|\cdot\|_{H^{1}(\Omega,h)}$.

The next result is also well known (see \cite[Theorem 2.8.1]{KJF}).
\begin{theorem}\label{teo2}
If $u_{n}\to u$ in $L^{p}(\Omega, h), 1<p<\infty$, then there exists a subsequence $\{u_{n_{k}}\}$ and a function $v\in L^{p}(\Omega, h)$ such that
\begin{enumerate}
\item[$(i)$] $u_{n_{k}}(x)\to u(x), \ n_{k}\to\infty, h-a.e. \ \mbox{on} \ \Omega$; \medskip
\item[$(ii)$] $|u_{n_{k}}(x)|\leq v(x), h-a.e. \ \mbox{on} \ \Omega$.
\end{enumerate}
\end{theorem}

Finally, we will enunciate a compact embedding type result for the weighted Sobolev spaces
$H^{1}(\Omega, h)$. See e.g. \cite{GU} for the details.

\begin{theorem}(Compact embeddings)\label{teo3}
Let $1\leq s\leq r<Nq/(N-q)$, $q\leq 2$ and
$$
K(h)=\max\left\{ |h^{-\frac{1}{2}}|_{L^{\frac{2q}{2-q}}(\Omega)}, |h^{\frac{1}{s}}|_{L^{\frac{rs}{r-s}}(\Omega)}\right\}<\infty.
$$
Then, the space $H^{1}(\Omega, h)$ is compactly embedded in $L^{s}(\Omega, h)$.
\end{theorem}

\section{Preliminaries: a problem (possibly) degenerate  on the boundary}\label{sec:prelim}
In this section, for future reference, we consider  the following  elliptic problem
\begin{equation}\label{PD}\tag{$P_{D}$}
\left \{ \begin{array}{ll}
- \textrm{div}(b(x)\nabla u)= f(u) & \mbox{in $D$,} \medskip \\
u=0 & \mbox{on $\partial D$,}
\end{array}\right.
\end{equation}
where $D\subset \R^{N}$ is a smooth, open and bounded  domain, $b\in C(\overline{D})$, $b(x)>0$ for  
$x\in D$, $b\in A_{2}$ and $1/b\in L^{t}(D), t>N/2$
and $f$ which satisfies \eqref{f1} and \eqref{f2} (with of course
$\overline \Omega$ replaced by $D$ and $a$
by the function $b$).
The operator $L(u)= -\textrm{div}(b(x)\nabla u)$ is called also $b-$elliptic.

A weak solution for \eqref{PD} is a function $u_\ast\in W^{1, 1}_{0}(D)\cap L^{\infty}(D)$ such that
\begin{equation}\label{2}
\int_{D}b(x)\nabla u_\ast\nabla v =\int_{D}f(u_\ast)v , \ \ \forall  v\in C_{c}^{\infty}(D).
\end{equation}
Observe that $b$ may eventually be zero somewhere on the boundary $\partial D$.

We will find the solution of \eqref{PD} working in the space $H^{1}_{0}(D,b).$
Note that such a space is contained into $W_{0}^{1, 2t/(1+t)}(D)$, where $t>N/2$ is given in \eqref{a2},
and then $2t/(1+t)>1$. 
Indeed, for $u\in H^{1}_{0}(D,b)$ we have, by the H\"older inequality,
\begin{eqnarray*}
\int_{D}|\nabla u|^{2t/(t+1)} &=& \int_{D}\left(\frac{1}{b^{t/(t+1)}}\right)(b^{t/(t+1)}|\nabla u |^{2t/(t+1)}) \\
&\leq& \left|\frac{1}{b}\right|^{t/(1+t)}_{L^{t}(D)}\|u \|_{H_{0}^{1}(D, b)}^{2t/(1+t)}<\infty
\end{eqnarray*}
and hence $H^{1}_{0}(D,b)\hookrightarrow W_{0}^{1, 2t/(1+t)}(D)$.
As it follows by the next proof, the solution will be also bounded.

The main result of this section is as follows.

\begin{theorem}\label{teo4}
Under the previous assumptions,
 problem \eqref{PD} has at least a nonnegative and nontrivial weak solution.
\end{theorem}

\begin{proof}
Due to the possibly degenerate structure of the problem,
 the suitable functional setting to treat \eqref{PD} is the weighted Sobolev space $H_{0}^{1}(D, b)$. Let $J: H_{0}^{1}(D, b)\to \R$ be the functional
$$
J(u)=\frac{1}{2}\int_{D}b(x)|\nabla u|^{2}-\int_{D}F_{\ast}(u) =:\frac{1}{2}\|u\|^{2}_{H_{0}^{1}(D, b)}-\psi(u),
$$
where $F_{\ast}$ is the primitive of
\begin{equation*}
f_{\ast}(s)=\left \{ \begin{array}{ll}
f(-\beta_{\ast}) & \mbox{ if  \ $s\in (-\infty, -\beta_{\ast}]$,}\\
f(s) & \mbox{if \ $s\in (-\beta_{\ast}, s_{\ast})$,}\\
0 & \mbox{ if  \ $s\in [s_{\ast}, \infty)$,}\\
\end{array}\right.
\end{equation*}
for some $\beta_{\ast}>0$ such that $f>0$ in $[-\beta_{\ast}, 0)$.

Observe that $J$ is well defined in $H_{0}^{1}(D, b)$. In fact, since 
 $f_{\ast}$ is bounded, we have
\begin{equation}\label{3}
\int_{D}|F_{\ast}(u)| \leq C\int_{D}|u|,
\end{equation}
for some positive constant $C$. On the other hand, by H\"older inequality and being $b\in A_2$
\begin{equation}\label{4}
\int_{D}|u|=\int_{D}\frac{1}{b^{1/2}}(b^{1/2}|u|)\leq \left|\frac{1}{b}\right|_{L^{1}(D)}^{1/2}\|u\|_{H_{0}^{1}(D, b)}<\infty,
\end{equation}
for all $u\in H_{0}^{1}(D, b)$. Moreover, since $f_{\ast}$ is continuous, $J\in C^{1}$.

Observe that $J$ is coercive. Indeed, from \eqref{3} and \eqref{4} we deduce
$$
J(u)\geq \frac{1}{2}\|u\|_{H_{0}^{1}(D, b)}^{2}-C\left|\frac{1}{b}\right|_{L^{1}(\Omega)}^{1/2}\|u\|_{H_{0}^{1}(\Omega, b)}.
$$

To prove that $J$ is weakly lower semicontinuous, it is enough to note that if $u_{n}\rightharpoonup u$ in $H_{0}^{1}(D, b)$, then, by \eqref{a2} and being $b\in C(\overline{D})$ the number $K(b)$ in Theorem \ref{teo3} is finite if we choose 
$$q=\frac{2t}{t+1}, \ s=2, \ \ r\in \Big(2, \frac{2Nt}{N(t+1)-2t}\Big).$$
 Therefore, by the compact embedding, we get
$$
u_{n}\to u \ \mbox{\ in $L^{2}(D, b)$}.
$$
From Theorem \ref{teo2}, up to a subsequence, there exists $g\in L^{2}(D, b)$ such that
$$
u_{n}(x)\to u(x) \ \mbox{ and } \  |u_{n}(x)|\leq g(x), \ \ b-\text{a.e. in } D.
$$
Since $b$ is positive in $D$, we obtain
$$
u_{n}(x)\to u(x) \ \mbox{ and  } \ |u_{n}(x)|\leq g(x),  \ \  \text{a.e. in } D.
$$
Consequently,
$$
F_{\ast}(u_{n}(x))\to F_{\ast}(u(x)) \ \ \mbox{a.e. in } D
$$
and
$$
|F_{\ast}(u_{n}(x))|\leq C|u_{n}(x)|\leq Cg(x)  \ \ \mbox{a.e. in } D.
$$
On the other hand, from $\eqref{a2}$
$$
\int_{D}|g(x)| = \int_{D} \frac{1}{b(x)^{1/2}} (b(x)^{1/2}|g(x)|) \leq \left|\frac{1}{b}\right|_{L^{1}(D)}^{1/2}|g|_{L^{2}(D, b)}<\infty,
$$
showing that $g\in L^{1}(D)$. Then by using the Lebesgue dominated convergence theorem, we conclude that
$$
\psi(u_{n})=\int_{D}F_{\ast}(u_{n})\to \int_{D}F_{\ast}(u)=\psi(u).
$$
Thus, $\psi$ is weakly continuous and, consequently, $J$ is weakly lower semicontinuous
 in the Hilbert space $H_{0}^{1}(D, b)$. 
Let $u_{\ast}:\Omega\to\R$ a minimum point of $J$. Since $J$ is $C^{1}$,
$$
\int_{D}b(x)\nabla u_{\ast}\nabla v =\int_{D}f_\ast(u_{\ast})v , \ \ \forall  v\in H_{0}^{1}(D, b),
$$
showing that $u_{\ast}$ is a weak solution of \eqref{PD}.

Now we are going to prove that $u_{\ast}$ is nontrivial. For that, it is enough to realise that 
$J$ takes negatives values.
Indeed let  $e_1$ be a positive eigenfunction associated to the first eigenvalue $\lambda_1(D)$ of Laplacian operator 
in $D$ with homogeneous Dirichlet boundary condition  and consider
$$
\frac{1}{s^{2}}J(se_{1})=\frac{1}{2}\|e_1\|^{2}_{H_{0}^{1}(D, b)}-\int_{D}\frac{F_{\ast}(se_{1})}{(se_1)^{2}} e_{1}^{2},
$$
for each $s>0$. By \eqref{f2}, de L'Hospital rule and Lebesgue dominated convergence theorem, by passing to the limit as $s\to 0^{+}$, we obtain
$$
\lim_{s\to 0^{+}}\frac{1}{s^{2}}J(se_{1})=\frac{1}{2}\int_{D}\left(b(x)-\frac{\gamma}{\lambda_1(D)}\right)|\nabla u|^{2}<0.
$$
Thus, for $s>0$ small enough, we have $J(u_{\ast})\leq J(se_1)<0$, showing that $u_{\ast}$ is nontrivial. 

It follows from \eqref{f1} and the definition of $f_\ast$ that by choosing 
$v=u_{\ast}^{-}:=\min\{u_{\ast}, 0\}$ 
in \eqref{2}, we have
$$
\int_{D}b(x)|\nabla u^{-}_{\ast}|^{2} =\int_{D}f_\ast(u_{\ast})u_{\ast}^{-} \leq 0.
$$
Since $b>0$, we conclude that $\nabla u^{-}_{\ast}=0$ a.e. in $D$. Therefore $u^{-}_{\ast}=c$ a.e. in $D$, for some $c\in\R$. Finally, from $u_{\ast}\in H_{0}^{1}(D, b)$, we have that  $u_{\ast}^{-}=0$ and 
$u_{\ast}=u_{\ast}^{+}:=\max\{u_{\ast}, 0\}\geq 0$. To conclude that $u_\ast\leq s_\ast$, it is enough to choose $v=(u_{\ast}-s_\ast)^{+}$ in \eqref{2} and reasoning in a similar way. Therefore $f_\ast(u_\ast)=f(u_\ast)$,
concluding the proof.
\end{proof}

%

\section{Proof of Theorem \ref{th:main}}\label{sec:main}

Finally we are ready to treat the problem
\begin{equation}\label{P}\tag{P}
\left \{ \begin{array}{ll}
-\textrm{div}(a(x)\nabla u)= f(u) & \mbox{in $\Omega$,}\medskip \\
u=0 & \mbox{on $\partial\Omega\cup a^{-1}\{0\}$}
\end{array}\right.
\end{equation}
and prove Theorem \ref{th:main}.

To take advantage of the degeneracy of $a$ in order to prove existence of multiple solutions to problem \eqref{P}, we will divide the proof in two steps. In the first one will be considered a suitable class of problems \eqref{Pil} with diffusion operator involving coefficients degenerating on the boundary of the domain where the problem is settled, that is, for each  $i\in\{1, \ldots, m\}$ and $l\in\{1, \ldots, j_i\}$, we will look for weak solutions of the problem
\begin{equation}\label{Pil}\tag{$P_{i,l}$}
\left \{ \begin{array}{ll}
-\textrm{div}(a(x)\nabla u)= f(u) & \mbox{in $\mathcal{A}^{(i)}_l$,}\\
u=0 & \mbox{on $\partial \mathcal{A}^{(i)}_l$.}
\end{array}\right.
\end{equation}

In the second one, the solutions obtained for \eqref{Pil} will be used to construct solutions to \eqref{P}, which have different numbers of positive bumps.

\medskip

{\bf Step I:} Existence of $\chi$ one-bump weak solutions to \eqref{Pil}.

\medskip

It follows from \eqref{a1} that each set $\mathcal{A}^{(i)}_l$ is a bounded domain of $\R^{N}$ with a smooth boundary, on which 
function $a$ can be zero. Consequently, Step I is a straightforward consequence of hypotheses \eqref{a2},\eqref{f1},\eqref{f2} 
 and Theorem \ref{teo4} in previous Section. Let us call $u_{i,l}$ the one-bump weak solution obtained to \eqref{Pil}.

\medskip

{\bf Step II:} Existence of $2^{\chi}-1$ nonnegative (and nontrivial) weak solutions to \eqref{P}.

\medskip

Let us consider the extensions $\widetilde{u}_{i, l}$ of $u_{i, l}$ to $\Omega$, that is,
$$
\widetilde{u}_{i, l}(x)=\left \{ \begin{array}{ll}
u_{i, l} & \mbox{in $\mathcal{A}^{(i)}_l$,}\\
0 & \mbox{in $\Omega\backslash\mathcal{A}^{(i)}_l$.}
\end{array}\right.
$$
Since $0\leq u_{i, l}\leq s_{\ast}$, $u_{i, l}\in H_{0}^{1}(\mathcal{A}^{(i)}_l, a_{|_{\mathcal{A}^{(i)}_l}})$ and
$$
\int_{\mathcal{A}^{(i)}_l}|\nabla u_{i, l}| = \int_{\mathcal{A}^{(i)}_l}\left(\frac{1}{a(x)^{1/2}}\right)(a(x)^{1/2}|\nabla u_{i, l}|) \leq \left|\frac{1}{a}\right|^{1/2}_{L^{1}(\mathcal{A}^{(i)}_l)}\|u_{i, l}\|_{H_{0}^{1}(\mathcal{A}^{(i)}_l, a_{|_{\mathcal{A}^{(i)}_l})}}^{2}<\infty,
$$
where in the last inequality we have used the Holder inequality. It is clear that $\widetilde{u_{i}}\in W_{0}^{1,1}(\Omega)\cap L^{\infty}(\Omega)$. Moreover, since $a\in C(\overline{\Omega})$ and $\mathcal{A}^{(i)}_l\subset \Omega\backslash a^{-1}\{0\}$, if $v\in C_{c}^{\infty}(\Omega\backslash a^{-1}\{0\})$ then  $v_{|_{\mathcal{A}^{(i)}_l}}\in H_{0}^{1}(\mathcal{A}^{(i)}_l, a_{|_{\mathcal{A}^{(i)}_l}})$. Thus, since $u_{i, l}$ is a weak solution of \eqref{Pil}, for all 
$v\in C_{c}^{\infty}(\Omega\backslash a^{-1}\{0\})$:
$$
\int_{\Omega}a(x)\nabla\left( \sum_{i, l}\widetilde{u}_{i, l}\right)\nabla v =\sum_{i, l}\int_{\mathcal{A}^{(i)}_l}a(x)\nabla u_{i, l}\nabla (v_{|_{\mathcal{A}^{(i)}_l}}) =\sum_{i, l}\int_{\mathcal{A}^{(i)}_l}f(u_{i, l})v_{|_{\mathcal{A}^{(i)}_l}} =\int_{\Omega}f\left(\sum_{i, l}\widetilde{u}_{i, l}\right)v,
$$
where the summation $\sum_{i, l}$ runs over all the possible combinations of indexes $i, l$, so as to include all the connected components of $\Omega\backslash a^{-1}\{0\}$, showing that $\widetilde{u}_{i, l}$ is a nonnegative and nontrivial weak solution of \eqref{P} for each $i\in\{1, \ldots, m\}$ and $l\in\{1, \ldots, j_i\}$. Since the sum of  $n$ of the previous weak solutions $\widetilde{u}_{i, l}$ ($2\leq n\leq \chi$) is still a solution of \eqref{P} (by \eqref{a1}), the result follows.
Observe finally that, arguing as in Section \ref{sec:prelim},
 the solutions found are in $H^{1}_{0}(\Omega, a)$.

%
%
%

%
%

\end{document}